\documentclass[12pt]{amsart}

\usepackage[colorlinks=true,linkcolor=blue]{hyperref}
\usepackage{amsmath}
\usepackage{amssymb}
\usepackage{amsthm}
\usepackage{mathtools}
\usepackage[shortlabels]{enumitem}
\usepackage{graphicx}
\usepackage{epstopdf}
\usepackage{epsfig}

\usepackage{enumitem}

\usepackage{hyperref}

\usepackage{times}
\usepackage{relsize}
\usepackage{textcomp}
\usepackage{amssymb}
\usepackage[english]{babel}
\usepackage[autostyle]{csquotes}
\usepackage{epstopdf}
\usepackage{nicefrac}
\usepackage{tikz}

\theoremstyle{plain} 
\newtheorem{thm}{Theorem}[section]
\newtheorem{prop}[thm]{Proposition}
\newtheorem{lemma}[thm]{Lemma}
\newtheorem{cor}[thm]{Corollary} 
\newtheorem{question}[thm]{Question}

\theoremstyle{remark}
\newtheorem{remark}[thm]{Remark}

\newtheorem{example}[thm]{Example}
\newtheorem{notation}[thm]{Notation}
\theoremstyle{definition}
\newtheorem{defin}[thm]{Definition}

\newcommand{\vF}{{ \mathbb F }}
\newcommand{\vP}{{ \mathbb P }}

\DeclareMathOperator{\Char}{char}

\begin{document}

\title{Square patterns in dynamical orbits}
\date{}
\author[Goksel]{Vefa Goksel}
\author[Micheli]{Giacomo Micheli}
\address{Towson University \\ Towson, MD 21252 \\ USA}
\email{vgoksel@towson.edu}
\address{University of South Florida\\ Tampa, FL 33620\\ USA}
\email{gmicheli@usf.edu}

\maketitle

\begin{abstract}
Let $q$ be an odd prime power. Let $f\in\vF_q[x]$ be a polynomial having degree at least $2$, $a\in \vF_q$, and denote by $f^n$ the $n$-th iteration of $f$. Let $\chi$ be the quadratic character of $\vF_q$, and $\mathcal{O}_f(a)$ the forward orbit of $a$ under iteration by $f$. Suppose that the sequence $(\chi(f^n(a)))_{n\geq 1}$ is periodic, and $m$ is its period. Assuming a mild and generic condition on $f$, we show that, up to a constant, $m$ can be bounded from below by $|\mathcal{O}_f(a)|/q^\frac{2\log_{2}(d)+1}{2\log_2(d)+2}$. More informally, we prove that the period of the appearance of squares in an orbit of an element provides an upper bound for the size of the orbit itself. Using a similar method, we can also prove that, up to a constant, we cannot have more than $q^\frac{2\log_2(d)+1}{2\log_2(d)+2}$ consecutive squares or non-squares in the forward orbit of $a$. In addition, we provide a classification of all polynomials for which our generic condition does not hold. 

\end{abstract}

\section{Introduction}
\label{sec:Intro}
Let $K$ be a field, and $f\in K[x]$ a polynomial with $d:=\deg(f)\geq 2$. We denote by $f^n$ the iterates of $f$ by self-composition, where $f^0 = x$, and $f^n:=f\circ f^{n-1}$ for $n\geq 1$. One fundamental object in dynamics is \emph{forward orbit} of an element $a\in K$, which is given by the set
\[\mathcal{O}_f(a) = \{a, f(a), f^2(a), \dots\}.\]
When this set is infinite, $a$ is called \emph{wandering point} of $f$. When this set is finite, $a$ is called a \emph{preperiodic point} of $f$. The latter case comes in two flavors: $a$ is called \emph{periodic} if $f^n(a) = a$ for some $n\geq 1$, and the smallest such positive integer $n$ is called the \emph{exact period} of $a$. If $f^{m+n}(a) = f^m(a)$ for some smallest non-negative integers $m,n$, then $a$ is called a \emph{preperiodic point of period $n$ and tail size $m$.}\par
There has been a lot of work on the orbit of an element $a\in\mathbb F_q$ under a polynomial map. See, for instance, \cite{AkbaryGhioca08,BGHKST13,Chang et al18,Heath-Brown,Hindes23,JMT16,ostafe2010length,Shao15,Sil08} for a limited list. See also \cite{Benedetto19} for a general overview of the dynamics over finite fields. This subject attracts researchers not only because of its theoretical interest, but also because of its applications to Pollard rho algorithm for factoring \cite{Bach91}.\par
In this paper, we have two main results on dynamical orbits. They concern dynamical orbits over finite fields and the occurrence of squares in dynamical orbits. In particular, we are able to bound the size of a dynamical orbit in terms of the period of its square elements, and also to prove that the number of consecutive squares in a dynamical orbit cannot be large.

Note that in the special scenario that the entire orbit consists of non-square elements and the polynomial is stable (i.e. all of its iterates are irreducible), a version of this problem was studied in \cite{ostafe2010length}.


Our main results concerning dynamical orbits over finite fields are the following.

\begin{thm}
\label{thm:orbitsize_explicit}
Let $q=p^k$ for some odd prime $p$ and $k\geq 1$, and let $d:=\deg(f)\geq 2$. Suppose that $f\in \mathbb{F}_q[x]$ is \textbf{not} in one of the following forms.
\begin{enumerate}[(a)]
\item $f=A(x-B)^{p^e}$ for some $A,B\in \vF_q$, $e\geq 1$.
\item $f = Ag^2$ for some $A\in \vF_q$, $g\in \vF_q[x]$.
\item $f = Axg^2$ for some $A\in \vF_q$, $g\in \vF_q[x]$.
\item $f=Ah^2+B$, where $A, B\in \vF_q$, and $h=\sum_{i=0}^{n}a_ix^i$ satisfies
\[a_0=\pm \sqrt{-\frac{B}{A}},\text{ } i(2i-1)Ba_i = -2(n+i-1)(n-i+1)a_{i-1}\] for $i\in \{1,2,\dots, n\}$.
\item $f=A(x-B)g^2$, where $A, B\in \vF_q$, and $g=\sum_{i=0}^{n}a_ix^i$ satisfies
\[\pm a_0 = \sqrt{-\frac{1}{A}}, \text{ }i(2i-1)Ba_i = -2(n-i+1)(n+i)a_{i-1}\]
for $i=1,2,\dots, n$.
\end{enumerate}
Let $a\in \mathbb{F}_q$. Suppose that the sequence $(\chi(f^n(a)))_{n\geq 0}$ is periodic, and let $m:=m_a$ be its period. Then
\[|\mathcal{O}_f(a)|=O\left(mq^\frac{2\log(d)+1}{2\log(d)+2}\right),\]
and the implied constant is only dependent on $d$.
\end{thm}

\begin{remark}
Two families of polynomials in parts (d) and (e) of Theorem~\ref{thm:orbitsize_explicit} arise from solutions of certain second order, linear differential equations (see (\ref{eq:second_order_linear1}) and (\ref{eq:g-g'-g''})). Perhaps interestingly, if $p\geq d$, they can be proven to be $\vF_q$-conjugate to $(-1)^dT_d(x)$, where $T_d$ is the Chebyshev polynomial of degree $d$. See Proposition~\ref{prop:PCF1} and Proposition~\ref{prop:PCF2}.
\end{remark}

Roughly speaking, Theorem~\ref{thm:orbitsize_explicit} implies that if the orbit is large, then the sequence of squares cannot obey a recurrence of low order. For example, if all elements in $\mathcal{O}_f(a)$  are squares (so $m=1$), and $f$ has degree $2$, then we have $|\mathcal{O}_f(a)|=O(q^\frac{3}{4}) $. See Section~\ref{sec:OrbitSize} for more details.

Moreover, we can also prove a bound on the number of consecutive squares of the forward orbit of any element of $\vF_q$, and for any polynomial $f$ that is not in one of the exceptional classes listed in Theorem~\ref{thm:orbitsize_explicit}.

Note that when we say $n$ \enquote{consecutive elements in $\mathcal{O}_f(a)$}, we mean $n$ elements in $\mathcal{O}_f(a)$ which are of the form $f^i(a), f^{i+1}(a)\dots, f^{i+n-1}(a)$ for some $i\geq 0$.

\begin{thm}
	\label{thm:consecutive_squares_explicit}
	Let $q=p^k$ for some odd prime $p$ and $k\geq 1$, and let $d:=\deg(f)\geq 2$. Suppose that $f\in \mathbb{F}_q[x]$ is \textbf{not} in one of the forms listed in (a)-(e) of Theorem~\ref{thm:orbitsize_explicit}.
	Let $a\in \mathbb{F}_q$.  Then the longest sequence of consecutive squares (or non-squares) in $\mathcal{O}_f(a)$ has length $O\left(q^\frac{2\log(d)+1}{2\log(d)+2}\right)$, where the implied constant is only dependent on $d$.
\end{thm}
The explicit forms of polynomials which are excluded from the statements of Theorem~\ref{thm:orbitsize_explicit} and Theorem~\ref{thm:consecutive_squares_explicit} come from the classification of polynomials over finite fields whose iterates satisfy a generic factorization property, which we now define.
\begin{defin}
Let $K$ be a field. A polynomial $f\in K[x]$ of degree at least $2$ is called \emph{dynamically ordinary} if for all $n\geq 1$, there exists an irreducible factor $g_n\in K[x]$ of $f^n$ such that $g_n$ does not divide $f^i$ for any $0\leq i<n$. More generally, for any $i\geq 1$, $f$ is called \emph{dynamically $i$-ordinary} if there exists such $g_n$ with multiplicity $m_n$ such that $i\nmid m_n$ for all $n\geq 1$.
\end{defin}
For example, if for all $n\geq 1$, $f^n$ has an irreducible factor $g_n\in K[x]$ with odd multiplicity such that $g_n \nmid f^i$ for any $0\leq i<n$, then $f$ becomes dynamically $2$-ordinary.

We chose the terminology \emph{dynamically ordinary} as such for the following reason: Based on the existing results and conjectures in the area (see, for instance, \cite[Conjecture 1.2]{LJ17}), almost all polynomials are expected to satisfy a stronger property called \emph{eventual stability}, which means that the number of irreducible factors of iterates of $f$ stabilizes after some $n$. We note that there is very little known about this conjecture, and it appears out of reach in most cases. See \cite{DS22,DHJMSS20,LJ17} for some partial results.\par

In this work, we proved new results about dynamically ordinary/$2$-ordinary polynomials over a general field $K$ (see Proposition~\ref{thm:powermap}, Theorem~\ref{thm:newirr}, Theorem~\ref{thm:strong_even} and Theorem~\ref{thm:strong_odd}), which lead to a complete classification of dynamically $2$-ordinary polynomials when $K$ is a finite field odd characteristic. 
\begin{thm}
\label{thm:2-ordinary}
Let $\mathbb{F}_q$ be a finite field and $f\in \mathbb{F}_q[x]$ a polynomial with $\deg(f)=d\geq 2$. Let $\Char(\mathbb{F}_q)=p$ for some odd prime $p$. 
Then, $f$ is not dynamically $2$-ordinary if and only if one of the following conditions hold:
\begin{enumerate}[(a)]
	\item $f=A(x-B)^{p^e}$ for some $A,B\in \vF_q$, $e\geq 1$.
	\item $f = Ag^2$ for some $A\in \vF_q$, $g\in \vF_q[x]$.
	\item $f = Axg^2$ for some $A\in \vF_q$, $g\in \vF_q[x]$.
	\item $f=Ah^2+B$, where $A, B\in \vF_q$, and $h=\sum_{i=0}^{n}a_ix^i$ satisfies
	\[a_0=\pm \sqrt{-\frac{B}{A}},\text{ } i(2i-1)Ba_i = -2(n+i-1)(n-i+1)a_{i-1}\] for $i\in \{1,2,\dots, n\}$.
	\item $f=A(x-B)g^2$, where $A, B\in \vF_q$, and $g=\sum_{i=0}^{n}a_ix^i$ satisfies
	\[\pm a_0 = \sqrt{-\frac{1}{A}}, \text{ }i(2i-1)Ba_i = -2(n-i+1)(n+i)a_{i-1}\]
	for $i=1,2,\dots, n$.
\end{enumerate}
\end{thm}

The structure of the paper is as follows: In Section~\ref{sec:Ordinary}, we classify all dynamically ordinary polynomials over finite fields. We then use the results of Section~\ref{sec:Ordinary} to classify all dynamically $2$-ordinary polynomials over finite fields odd characteristic in Section~\ref{sec:strong}. Finally, in Section~\ref{sec:OrbitSize},  we use our results on dynamically $2$-ordinary polynomials together with Weil's bound for character sums to prove Theorem~\ref{thm:orbitsize_explicit} and Theorem~\ref{thm:consecutive_squares_explicit}. 

\subsection*{Acknowledgements} We thank Rob Benedetto for helpful comments related to the material in this paper. We also thank Andrea Ferraguti for a helpful feedback on an earlier draft of the paper. G.Micheli was partially supported by NSF grant number 2127742.

\section{Dynamically ordinary polynomials}
\label{sec:Ordinary}
Let $K$ be a field. Recall that we call a polynomial $f\in K[x]$ \emph{dynamically ordinary} if for all $n\geq 1$, there exists an irreducible factor $g_n$ of $f^n$ such that $g_n\nmid f^i$ for any $0\leq i< n$. In this section, we will first give a sufficient condition for a polynomial of degree at least $2$ over a general field to be dynamically ordinary, and then will use this to classify all dynamically ordinary polynomials over finite fields.\par 
Our first result gives a sufficient condition for a polynomial to be dynamically ordinary over a finite field.
For $n\geq 0$ and $\alpha\in K$, define the set $R_{n,\alpha}(f)$ by
\[R_{n,\alpha}(f) = \{\beta\in \overline{\mathbb{F}}_q\text{ }| f^n(\beta)=\alpha\}.\]
We first recall that the preimages of a point $\alpha\in K$ under iteration by $f$ lead to an infinite rooted tree, which is a fundamental object of study in arithmetic dynamics. The presence of this tree structure will be convenient for our purposes as well. 
\begin{defin}
	\label{def:tree}
	Let $K$ be a field, and $f\in K[x]$ a polynomial with $\deg(f)=d\geq 2$. Let $\alpha\in \overline{K}$. Define the set $T_{\alpha}(f)$ by the disjoint union
	\[T_{\alpha}(f) = \sqcup_{i=0}^{\infty} R_{i,\alpha}(f).\]
	$T_{\alpha}(f)$ has a natural tree structure, as follows: For any $i\geq 0$, draw an edge between $\beta\in R_{i+1,\alpha}(f)$ and $\theta\in R_{i,\alpha}(f)$ if and only if $f(\beta) = \theta$. We call $T_{\alpha}(f)$ \emph{$f$-preimage tree of $\alpha$}. We say that $T_{\alpha}(f)$ is \emph{repeating} if there exist distinct non-negative integers $n,m$ such that $R_{n,\alpha}(f)\cap R_{m,\alpha}(f)\neq \emptyset$. 
\end{defin}
\begin{lemma}\label{lemma:onlyone}
Let $K$ be a field and $f\in K[x]$ a polynomial with $\deg(f)=d\geq 2$.
Let $\alpha_1,\alpha_2$ be distinct roots of $f$.
Suppose that $T_{\alpha_1}(f)$ is repeating. Then $T_{\alpha_2}(f)$ is not repeating. Moreover, $0$ is periodic, $0\in T_{\alpha_1}(f)$, and $0\not\in T_{\alpha_2}(f)$.
\end{lemma}
\begin{proof}
Suppose for the sake of contradiction that the trees $T_{\alpha_1}(f)$ and $T_{\alpha_2}(f)$ are both repeating. Thus, for any $j\in\{1,2\}$, there exist two positive integers $m_j,n_j$, with $m_j<n_j$ such that 
$R_{n_j-1,\alpha_j}(f)\cap R_{n_j-m_j-1,\alpha_j}(f)\neq \emptyset$. Let $\beta_j\in R_{n_j-1,\alpha_j}(f)\cap R_{n_j-m_j-1,\alpha_j}(f)$.

This forces that 
\begin{equation}
\label{eq:beta_repeat}
f^{n_j}(\beta_j) = f^{n_j-m_j}(\beta_j) = 0.
\end{equation}

Using \eqref{eq:beta_repeat} immediately yields
\[0=f^{n_j}(\beta_j) = f^{m_j}(f^{n_j-m_j}(\beta_j)) = f^{m_j}(0),\]
thus $0$ is periodic with period dividing $m_j$. By definition of $\beta_j$ and using \eqref{eq:beta_repeat}, we obtain
\begin{equation}
\label{eq:formula_alpha_j}
\alpha_j = f^{n_j-1}(\beta_j) = f^{m_j-1}(f^{n_j-m_j}(\beta_j)) = f^{m_j-1}(0).
\end{equation}
We now let $t$ be the exact period of $0$. It follows that $m_j = tk_j$. Combining this with \eqref{eq:formula_alpha_j}, we obtain
\[\alpha_j = f^{m_j-1}(0) = f^{tk_j-1}(0) = f^{t-1}(f^{t(k_j-1)}(0)) = f^{t-1}(0),\]
where we used the fact that $0$ has exact period $t$ in the last equality.
Since $t$ does not depend on $j$, this forces $\alpha_1=\alpha_2$, which is a contradiction. From this we also deduce that $0$ cannot lie in $T_{\alpha_2}(f)$, or otherwise $f^{t-1}(0)=\alpha_2\neq \alpha_1$.
\end{proof}

\begin{prop}
\label{thm:newirr}
Let $K$ be a field and $f\in K[x]$ a polynomial with $\deg(f)=d\geq 2$. Suppose $f\neq A(x-B)^d$ for any $A,B\in K$. Then $f$ is dynamically ordinary.
\end{prop}
\begin{proof}
To prove Proposition~\ref{thm:newirr}, it suffices to prove the following statement: For any $n\geq 1$, there exist $\beta\in R_{n,0}(f)$ such that $\beta\notin R_{i,0}(f)$ for $i<n$. We will now prove this statement.\par

By assumption on $f$, the statement clearly holds for $n=1$. We can now fix some $n\geq 2$. Note that if $\alpha_1,\alpha_2,\dots, \alpha_d\in R_{1,0}(f)$ (not necessarily distinct), we have
\[R_{n,0}(f) = f^{-(n-1)}(\alpha_1)\cup f^{-(n-1)}(\alpha_2)\cup\dots \cup f^{-(n-1)}(\alpha_d).\]

Since $f\neq A(x-B)^d$, $f$ has at least two distinct roots, so we can use Lemma \ref{lemma:onlyone} to establish that there is at least one non-repeating subtree  $T_{\alpha}(f)$ and this subtree does not contain the zero. Let $\alpha'$ be any other root of $f$. We will now show that for every $n\geq 1$, $R_{n,\alpha}(f)$  has roots that do not appear in $R_{m, \alpha'}(f)$ for any $m<n$, and therefore $f$ is dynamically ordinary:
Let $\beta$ be a root of $f^n-\alpha$ and of $f^m-\alpha'$. Then,
$\alpha=f^{n}(\beta)=f^{n-m}(f^{m}(\beta))=f^{n-m}(\alpha')$.
By applying $f$, we conclude that $0$ is periodic. Also, if $t$ is the period of $0$ we also get $f^{t-1}(0)=\alpha$, showing that $0$ is in the subtree $T_{\alpha}(f)$, which is a contradiction.
\end{proof}


In the rest of the section, we will study the polynomials of the form $f=A(x-B)^d\in K[x]$. This will eventually lead to a complete classification of dynamically ordinary polynomials when $K$ is a finite field. We start by slightly extending the definition of the notion \emph{dynamically ordinary}. 
\begin{defin}
Let $K$ be a field, and $f\in K[x]$ a polynomial with $\deg(f)=d\geq 2$. Let $\alpha\in K$. We say the pair $(f,\alpha)$ is \textit{dynamically ordinary} if for all $n\geq 1$, $f^n-\alpha$ has an irreducible factor in $K[x]$ that does not divide $f^i-\alpha$ for $i<n$. Similarly, we say the pair $(f,\alpha)$ is \textit{dynamically $2$-ordinary} if for all $n\geq 1$, $f^n-\alpha$ has an irreducible factor with odd multiplicity in $K[x]$ which does not divide $f^i-\alpha$ for $i<n$. 
\end{defin}
Recall that if $\alpha=0$, then we simply say $f$ is \textit{dynamically ordinary} (resp. \textit{dynamically $2$-ordinary}).
The following lemma will be crucial in the proof of our next result.
\begin{lemma}
\label{lem:conjugation}
Let $K$ be a field, and $f\in K[x]$ a polynomial with $\deg(f)=d\geq 2$. Let $\alpha\in K$, and $g(x) = ax+b \in PGL_2(K)$ such that $g(\alpha)=0$. Set $h=g\circ f\circ g^{-1}$. Then, $(f,\alpha)$ is dynamically ordinary if and only if $h$ is dynamically ordinary.
\end{lemma}
\begin{proof}
For any $i\geq 1$, it follows by direct computation that $\beta\in R_{i,\alpha}(f)$ if and only if $g(\beta)\in R_{i,0}(h)$.\par 

First assume that $(f,\alpha)$ is dynamically ordinary. For $n\geq 1$ and $\beta_1, \beta_2, \dots, \beta_k\in \overline{K}$, let $g_n = A_n\prod_{j=1}^{k} (x-\beta_j)\in K[x]$ be an irreducible factor of $f^n-\alpha$ such that $g_n\nmid f^i-\alpha$ for $i<n$. Therefore, we have \[\{g(\beta_1),g(\beta_2),\dots,g(\beta_k)\}\subseteq R_{n,0}(h).\] Since $\beta_1,\beta_2,\dots,\beta_k$ cannot lie in $R_{i,\alpha}(f)$ for $i<n$, it follows that $g(\beta_1), g(\beta_2),\dots, g(\beta_k)$ cannot lie in $R_{i,0}(h)$ for $i<n$. If we let $h_n = g_n\circ g^{-1}\in K[x]$, $h_n$ becomes an irreducible factor of $h^n$ which does not divide $h^i$ for $i<n$. This shows that $h$ is dynamically ordinary.\par
We now assume that $h$ is dynamically ordinary. For $n\geq 1$, let $h_n$ be an irreducible factor of $h^n$ such that $h_n\nmid h^i$ for $i<n$. Similar to the argument in the first part, it follows that the polynomial $g_n=h_n\circ g\in K[x]$ is an irreducible factor of $f^n-\alpha$ that does not divide $f^i-\alpha$ for $i<n$. The proof of Lemma~\ref{lem:conjugation} is now complete.
\end{proof}
We are now ready to give a sufficient condition for a polynomial of the form $A(x-B)^d$ over a general field to be dynamically ordinary.
\begin{thm}
\label{thm:powermap}
Let $K$ be a field and $f = A(x-B)^d\in K[x]$ for some $A,B\in K$, where $d\geq 2$. Suppose $B\neq 0$. For any positive integer $n$, set $Z_0=A$, $W_0=-B$,  $Z_n=AZ_{n-1}^d$, and $W_n=AW_{n-1}^d-B$.
Then $f$ is not dynamically ordinary if and only if the following conditions hold.
\begin{enumerate}
	\item $\Char(K)=p$ for some prime $p$, and $d=p^k$ for some $k\geq 1$.
	\item The sequence $H_n=W_n/Z_n$ repeats at least once.
\end{enumerate}
\end{thm}
\begin{remark}
If $B=0$, $f$ is trivially not dynamically ordinary. Also note that if $\text{char}(K) = 0$, Theorem~\ref{thm:powermap} implies that any polynomial $f\in K[x]$ is dynamically ordinary as long as $f$ is not of the form $Ax^d$.
\end{remark}
\begin{proof}
We have $R_1(f)=\{B\}$. Note that $f$ is dynamically ordinary if and only if for all $n\geq 2$, there is $\beta\in f^{-(n-1)}(B)$ with $\beta\notin f^{-j}(B)$ for $j<n-1$. It follows that $f$ is dynamically ordinary if and only if $(f,B)$ is dynamically ordinary. Let $h=gfg^{-1}$, where $g=x-B\in PGL_2(K)$. Using Lemma~\ref{lem:conjugation}, we conclude that
$f$ is dynamically ordinary if and only if is dynamically ordinary.
A direct calculation yields
\[h = Ax^d-B.\]
Thus, Proposition~\ref{thm:newirr} implies that $f$ is dynamically ordinary unless we have
\begin{equation}
\label{eq:oneroot}
Ax^d-B = A(x-X_1)^d
\end{equation}
for some $X_1\in K$. If $X_1=0$, then $B=0$, which contradicts the hypotheses on $f$. Therefore, $X_1\neq 0$. Now, 
$Ax^d-B$ has a single root if and only if $d=p^k$ is a power of the characteristic of $K$. In this case, we have that the $n$-th iteration of $f$ is $Z_n x^{d^n}+W_n$, 
which has $\sqrt[d^n]{-W_n/Z_n}$ as the only root. The claim follows immediately by observing that $\sqrt[d^n]{-(\cdot)}$ is a bijection of $\overline K$.

\end{proof}

Since $d=p^k$, we have that $\sqrt[d^n]{\cdot}$ is a bijection of any finite field of characteristic $p$. This allows us to give an explicit classification of dynamically non-ordinary polynomials when $K$ is finite.
\begin{cor}
\label{cor:finite_case}
Let $\mathbb{F}_q$ be a finite field and $f\in \mathbb{F}_q[x]$ a polynomial with $\deg(f)=d\geq 2$. Let $\text{char}(\mathbb{F}_q)=p$. Then, $f$ is not dynamically ordinary if and only if the following conditions hold:
\begin{enumerate}
\item $f=A(x-B)^d$ for $A,B\in \mathbb{F}_q$.	
\item $d=p^e$ for some $e\geq 1$. 
\end{enumerate}
\end{cor}
\begin{proof}
Use Theorem \ref{thm:powermap} with $K$ finite and observe that the sequence $H_n$ is necessarily repeating.
\end{proof}



\section{dynamically $2$-ordinary polynomials}
\label{sec:strong}
Let $K$ be a field. Recall that $f$ is \emph{dynamically $2$-ordinary} if it satisfies the following: For all $n\geq 1$, $f^n$ has an irreducible factor $g_n$ with odd multiplicity such that $g_n \nmid f^i$ for any $0\leq i<n$. In this section, we will give necessary and sufficient conditions for a dynamically ordinary polynomial to be dynamically $2$-ordinary, which will imply Theorem~\ref{thm:2-ordinary} using Corollary~\ref{cor:finite_case}.

We start with a technical lemma, which shows that if a polynomial is not dynamically $2$-ordinary, then its roots must satisfy some strong algebraic conditions. Lemma~\ref{lemma:special_form_str} will be crucial in the proofs of the results in this section.

\begin{lemma}\label{lemma:special_form_str}
Let $K$ be a field such that $\text{char}(K)\neq 2$, and $f\in K[x]$ be a polynomial with degree $d\geq 2$. Suppose that $f$ is not dynamically $2$-ordinary. If $d$ is odd, then $f$ has at most one root with odd multiplicity in $\overline K$. If $d$ is even, then there exists a root $B$ of $f$ with odd multiplicity such that $f-B=Ah^2$ for some $h\in K[x]$ and $A\in K$, and $f$ has exactly one more root with odd multiplicity.
\end{lemma}
\begin{proof}
First, observe that the problem is purely geometrical, so we can consider the problem for $K$ algebraically closed. If $f$ is a square, then there is nothing to do.  Let $\alpha,\beta$ be distinct roots of $f$ appearing with odd multiplicity.
Since at most one of the subtrees $T_{\alpha}(f)$ and $T_{\beta}(f)$ is repeating by Lemma  \ref{lemma:onlyone}, we can assume without any loss that $T_{\beta}(f)$ is non-repeating. Moreover,  the roots $z$ of $f^n-\beta$ can never appear as roots of $f^m-\theta$ for any root $\theta\neq \beta$ of $f$ and any $m<n$, or otherwise we get $f^m(z)-\theta=0$ and $f^n(z)-\beta=0$, which forces $f^n(z)=\beta=f^{n-m}(\theta)=f^{n-m-1}(0)$. But then this implies that $\beta$ is in the orbit of $0$ (and $0$ is periodic) which is impossible by Lemma \ref{lemma:onlyone} ($0$ cannot be present in a non-repeating subtree). Therefore, to deny dynamically $2$-ordinarity, we need all factors of $f^n-\beta$ to appear with even multiplicity for some $n\geq 1$, which means that $f^n-\beta$ must be a square. This concludes the proof for the odd degree case, since $f^n-\beta$ has odd degree for $n\geq 1$ and therefore can never be a square.
For $f$ of even degree, take now the smallest $n\geq 1$ for which $f^n-\beta$ is a square.
Any irreducible factor of $f^n-\beta$ is a divisor of $f-\beta_{n}$ for some root $\beta_n$ of $f^{n-1}-\beta$. But now $f^{n-1}-\beta$ is not a square, so there exists $\beta_n$ with odd multiplicity so that $f^n-\beta$ factors as $c(f-\beta_n)^{2e+1}g$, where $g$ is coprime to $f-\beta_n$ (because all factors $f-\gamma$ are coprime for distinct $\gamma$'s) and some $c\in K$. So, $f-\beta_n$ must be a square, say $h^2$. But now being a square is a geometric condition up to the squareness of the leading term of $h$, from which we obtain that $f=Ah^2+B$ for some $A,B\in K$.\par 

We now want to show that $B$ is a root of $f$. To see this, let $\alpha,\beta$ be roots of $f$. Recall that at least one of the two subtrees $T_{\alpha}(f)$ or $T_{\beta}(f)$ would be non-repeating by Lemma \ref{lemma:onlyone}. 
Suppose that $T_{\beta}(f)$ is non-repeating. Again, the only chance to deny dynamically $2$-ordinarity is that all new factors appear with even degree, from which it follows that $f^n-\beta$ is a square. But since we are in odd characteristic and $f=Ah^2+B$, this happens only if $B=\beta$. 

Now we want to conclude that there is exactly one other root with odd multiplicity. But this is immediate: Suppose that there are at least three distinct roots other than $B$ with odd multiplicity. One of them, say $\alpha$, has to lead to a non-repeating subtree $T_{\alpha}(f)$.  To deny dynamically $2$-ordinarity, we would again need $f^n-\alpha$ to be a square for some $n\geq 1$. Similar to the previous paragraph, this forces $\alpha=B$, a contradiction. Hence, there is exactly one root with odd multiplicity other than $B$, completing the proof of Lemma~\ref{lemma:special_form_str}.
\end{proof}
We are now ready to classify all dynamically ordinary polynomials $f\in K[x]$ that are dynamically $2$-ordinary. We will split the question into two cases based on whether $\deg(f)=d$ is even or odd.
\begin{thm}
	\label{thm:strong_even}
Let $K$ be a field such that $\text{char}(K)\neq 2$, and $f\in K[x]$ a dynamically ordinary polynomial with $\deg(f)=d\geq 2$. Suppose that $d$ is even. Then $f$ is not dynamically $2$-ordinary if and only if $f$ satisfies one of the following conditions: 
\begin{enumerate}
\item $f = Ag^2$ for some $A\in K$, $g\in K[x]$.
\item $f=Ah^2+B$, where $A, B\in K$, and $h=\sum_{i=0}^{n}a_ix^i$ satisfies
\[a_0=\pm \sqrt{-\frac{B}{A}},\text{ } i(2i-1)Ba_i = -2(n+i-1)(n-i+1)a_{i-1}\] for $i\in \{1,2,\dots, n\}$.
\end{enumerate}
\end{thm}
\begin{proof}
If $f=Ag^2$, $f$ is clearly not dynamically $2$-ordinary, so that case can be excluded. Now, using Lemma \ref{lemma:special_form_str}, we conclude that we can restrict to $f=Ah^2+B$ with $B$ a non-zero root of $f$.
First, we would like to prove that the only other root with odd multiplicity is zero, i.e. $f=Ax(x-B)g^2$ for some $g\in K[x]$.

As Lemma \ref{lemma:special_form_str} provides, the only admissible form of $f$ is $A(x-B)(x-\alpha_1)g^2$. Now, the pair $(f,\alpha_1)$ cannot be dynamically $2$-ordinary, as otherwise $f$ would be dynamically $2$-ordinary. Note that $f-\alpha_1 = Ah^2+B-\alpha_1$ cannot be a square because in odd characteristic this is possible only if $\alpha_1 =B$, which will make $f$ a square as well, contradicting the assumption. So, consider two distinct roots $\beta_1, \beta_2$ of $f-\alpha_1$ with odd multiplicity. By Lemma~\ref{lemma:onlyone}, we can assume without any loss that the subtree $T_{\beta_2}(f)$ is non-repeating. Similarly to how we argued in the proof of Lemma~\ref{lemma:special_form_str}, then, for $f$ to be not dynamically $2$-ordinary, $f^n-\beta_2$ must be a square for some $n\geq 1$. But, since $f$ is of the form $Ah^2+B$ and the characteristic is odd, we again conclude that $\beta_2 = B$. Hence, $B$ is both a root of $f$ and a root of $f-\alpha_1$, allowing us to conclude $\alpha_1 =0.$




Thus, we have
\begin{equation}
\label{eq:Cx(x-B)g^2}
f = Cx(x-B)g^2
\end{equation}
for some $C\in K$ and $g\in K[x]$. Recall that we also have
\begin{equation}
\label{eq:f-B=Ah^2}
	f-B = Ah^2
\end{equation}
for some $h\in K[x]$. Taking the derivatives of both equations and simplifying lead to
\begin{equation}
\label{eq:2Chh'}
Cg((2x-B)g+2x(x-B)g') = 2Ahh'.
\end{equation}
Let $\ell_g, \ell_h$ be the leading coefficients of $g$ and $h$, respectively. Using (\ref{eq:Cx(x-B)g^2}) and (\ref{eq:f-B=Ah^2}), assume without loss that $\frac{\ell_g}{\ell_h}=\sqrt{\frac{A}{C}}$. Note that $g$ and $h$ cannot have any common roots. Set $n=\frac{d}{2}$. By comparing leading coefficients in (\ref{eq:2Chh'}), then, we obtain 
\begin{equation}
\label{eq:g_constant_h'}
g= \frac{1}{n}\sqrt{\frac{A}{C}}h'.
\end{equation}
Using this together with (\ref{eq:Cx(x-B)g^2}) and (\ref{eq:f-B=Ah^2}), and simplifying yields 
\[n^2(h^2+\frac{B}{A}) = x(x-B)(h')^2.\]
Differentiating both sides gives
\begin{equation}
\label{eq:second_order_linear1}
2n^2h = (2x-B)h' + 2x(x-B)h''.
\end{equation}
This is a second order, linear differential equation. Let $h=\sum_{i=0}^{n} a_ix^i$, which gives $h'=\sum_{i=0}^{n-1}(i+1)a_{i+1}x^i$. Comparing the coefficients of both sides in (\ref{eq:second_order_linear1}) and simplifying yields the formula
\begin{equation}
\label{eq:recurrence_even}
i(2i-1)Ba_i = -2(n+i-1)(n-i+1)a_{i-1}
\end{equation}
for $i=1,2,\dots, n$. Recalling $f(0) = 0$ and by comparing constant coefficients of both sides in (\ref{eq:f-B=Ah^2}), we also have $a_0 = \pm \sqrt{-\frac{B}{A}}$, completing the proof of \emph{only if} direction. Finally, conditions (1) and (2) clearly imply that $f$ is not dynamically $2$-ordinary, finishing the proof of Theorem~\ref{thm:strong_even}.
\end{proof}
Next, we will show that the polynomials arising in Theorem~\ref{thm:strong_even} (2) are $K$-conjugate to $T_d$ when $\text{char}(K)=0$ or $\text{char}(K)\geq d$, where $T_d$ is the Chebyshev polyomial of degree $d$.
\begin{prop}
	\label{prop:PCF1}
	Let $K$ be a field such that either $\text{char}(K)=0$ or $\text{char}(K)\geq d$ and $\text{char}(K)$ is odd. Suppose that $f\in K[x]$ is defined as follows: $f=Ah^2+B$, where $A, B\in \vF_q$, and $h=\sum_{i=0}^{n}a_ix^i$ satisfies
	\[a_0=\pm \sqrt{-\frac{B}{A}},\text{ } i(2i-1)Ba_i = -2(n+i-1)(n-i+1)a_{i-1}\] for $i\in \{1,2,\dots, n\}$. Then $f$ is $K$-conjugate to $T_d(x)$.
\end{prop}
\begin{proof}
	 Let $b_i:= \frac{a_i}{a_0}B^i\sqrt{-1}$. 
	 Recalling $a_0=\pm\sqrt{-\frac{B}{A}}$, this immediately yields
	\[f = B\bigg(\bigg(\sum_{i=0}^{n} b_i\big(\frac{x}{B}\big)^i\bigg)^2+1\bigg),\]
	which we observe being independent of $A$, so we can select $A=-2$ without changing $f$. Since we only consider the question up to linear conjugacy, and $f$ is $K$-conjugate to $g:=(\sum_{i=0}^{n} b_ix^i)^2+1$ for any choice of $B$ (by $\gamma(x)=Bx\in \text{PGL}(2,K)$), we conclude that $B$ can be freely chosen and this does not affect the conjugacy class of $f$.
	
	Recall from the proof of Theorem \ref{thm:strong_even} that the polynomial $h$ is a solution of the differential equation
	\[2n^2h = (2x-B)h' + 2x(x-B)h''.\]
	Reorganizing and simplifying, we obtain
	\[((x-B/2)^2-(B/2)^2)h'' + (x-B/2)h' - n^2h = 0.\]
	We now specialize to $B=2$, which yields
	\[((x-1)^2-1)h''+(x-1)h'-n^2h = 0.\]
	This is the defining differential equation of the Chebyshev polynomial $T_n(x)$ shifted by $1$. By the form of the solutions of the differential equation and the assumption on $\text{char}(K)$, it is clear that the polynomial solution is unique up to a constant multiple. Hence, we obtain $h(x) = CT_n(x-1)$ for some $C\in K$. Remember that we fixed $A=-2$, so using this $h$ in the expression $f=-2h^2+B$, then, we get
	\begin{equation}
		\label{eq:chebyshev}
		f = -2(CT_n(x-1))^2+2
	\end{equation}
	Note that the constant coefficient of $T_n(x-1)$ is $T_n(-1) = 1$. Hence, $-2C^2+2=0$, which forces $C=\pm 1$.Using this in (\ref{eq:chebyshev}), we obtain
	\[f=-2(T_n(x-1))^2+2.\]
	Observe that $f$ is $K$-conjugate to $-2(T_n(x))^2+1$. Now, noticing that $T_2(x) = -2x^2+1$, it follows that $f$ is $K$-conjugate to $T_2\circ T_n = T_{2n}$, where the last equality is a well-known identity. (see \cite[Pg. 329, Proposition 6.6]{Sil07}) This completes the proof of Proposition~\ref{prop:PCF1}.
\end{proof}
\begin{thm}
	\label{thm:strong_odd}
	Let $K$ be a field such that $\text{char}(K)\neq 2$, and $f\in K[x]$ a dynamically ordinary polynomial with $\text{deg}(f)=d\geq 2$. Suppose that $d$ is odd. Then $f$ is not dynamically $2$-ordinary if and only if $f$ satisfies one of the following conditions: 
	\begin{enumerate}
		\item $f = Axg^2$ for some $A\in K$, $g\in K[x]$.
		\item $f=A(x-B)g^2$, where $A, B\in K$, and $g=\sum_{i=0}^{n}a_ix^i$ satisfies
		\[\pm a_0 = \sqrt{-\frac{1}{A}}, \text{ }i(2i-1)Ba_i = -2(n-i+1)(n+i)a_{i-1}\]
		for $i=1,2,\dots, n$.
	\end{enumerate}
\end{thm}
\begin{proof}

Again, we reason over the algebraic closure of $K$, as the problem is completely equivalent.
Suppose that $f$ is not dynamically $2$-ordinary. Since $f$ is clearly not dynamically $2$-ordinary when condition (1) holds, we can assume without any loss that $f$ is not of the form $Axg^2$. Using Lemma \ref{lemma:special_form_str} we recall that $f$ has at most one root $B$ with odd multiplicity.

If $T_B(f)$ were a non-repeating subtree, we can argue similarly to the proof of Lemma~\ref{lemma:special_form_str} to conclude that $f$ is dynamically $2$-ordinary. Hence, $T_B(f)$ is repeating, $0$ lies in $T_B(f)$, and $0$ is a periodic point of $f$ by Lemma \ref{lemma:onlyone}.

Let $\alpha_1$ be a root of $f-B$ appearing with odd multiplicity. Then $f-B=A (x-\alpha_1)k^2$ because otherwise there would be at least one root $\alpha$ of $f-B$ with odd multiplicity such that $T_{\alpha}(f)$ is non-repeating, from which we can again conclude that $f$ is dynamically $2$-ordinary.
Let us now show that $\alpha_1=0$, so that the full orbit of $0$ is $\{0,B\}$. To do this, we need to observe some facts about ramification of the covering of $\vP^1$ induced by $f$.

From the first paragraph, we have $f=A(x-B)g^2$ for some $A,B\in K$. Note that $g$ and $k$ cannot have a common root (unless $B$ is zero, which was excluded earlier). Now, if $g$ or $k$ has a root with multiplicity larger than or equal to $2$, using $abc$ theorem for function fields on the triple $(f, -f-B, B)$ gives a contradiction:
\[\deg(f)\leq ((d-1)/2+1)+((d-1)/2-1+1)-1 = d-1.\] 
So $g$ and $k$ have all distinct roots and no common root.

Now consider the covering given by the map $f$ on $\mathbb P^1$: Looking at all roots of $k$ and $g$, we observe that it has at least $d-1$ branch points over ramified image values (because these are in correspondence with the zeroes of $f'$), so we exhausted ramification at finite points (because $g,h$ are coprime) and for all other finite $\beta$ we must have $f-\beta$ unramified. This means that there cannot be an $\alpha_2$, different from $B$ and $\alpha_1$, for which $f-\alpha_1=(x-\alpha_2)g^2$ (again, $g$ cannot have roots in common with $h$ or $k$, or $f$ is not a well defined map). 

Now we are ready to prove that $\alpha_1=0$. Suppose the contrary.
Recall that $B$ is at the first level, and $\alpha_1$ is at the second level of the tree (since the image of $\alpha_1$ is $B$). Since we exhausted ramification with the preimages of $0$ and $B$, we are sure that $f-\alpha_1$ has no repeated root. We can now see that the zero can only lie in one subtree rooted at one of roots of $f-\alpha_1$, because otherwise if it is in two of them then $f-\alpha_1$ would have repeated roots: Since $0$ is periodic and $\alpha_1$ is in the orbit of zero (because by assumption $T_{\alpha_1}(f)$ is repeating, hence $0\in T_{\alpha_1}(f)$), the appearence of $0$ in two distinct subtrees (each rooted at a root of $f-\alpha_1$) at distance $t$ from $\alpha_1$ for some $t\geq 1$ would guarantee that $f^{t-1}(0)$ is a double root for $f-\alpha_1$. By Lemma~\ref{lemma:onlyone}, then, at most one of the subtrees rooted at roots of $f-\alpha_1$ would be repeating. Let $\theta$ be a root of $f-\alpha_1$ such that $T_{\theta}(f)$ is non-repeating. Note that $B$ is a root of $f$ with odd multiplicity that does not appear earlier in the tree, and $\alpha_1$ is also such a root of $f^2$. For $n\geq 3$, considering the roots of $f^{(n-2)}-\alpha_1$ (which are clearly roots of $f^n$) that lie in $T_{\theta}(f)$ would lead to roots of $f^n$ with multiplicity one which do not appear at any previous level. Therefore, $f$ is a $2$-ordinary polynomial, a contradiction. We conclude $\alpha_1=0$.


From the facts that $\alpha_1=0$ and $0$ is periodic, we conclude that $f(0)=B$, which leads to the following equations for $f$:
\begin{equation}
\label{eq:f = A(x-B)g^2}
f = A(x-B)g^2
\end{equation}
and
\begin{equation}
\label{eq:f-B=Cxh^2}
f-B=Cxh^2
\end{equation}
for some $A, B, C\in K$, $g,h\in K[x]$. Taking the derivatives of both equations and simplifying leads to
\begin{equation}
\label{eq:Ag(g+2(x-B)g') = Ch(h+2xh')}
Ag(g+2(x-B)g') = Ch(h+2xh').
\end{equation}
Noticing that $g$ and $h$ cannot have common factors, and that $\text{deg}(g) = \text{deg}(h)$, and (hence) 
\[\text{deg}(g+2(x-B)g') = \text{deg}(h+2xh')\] by (\ref{eq:Ag(g+2(x-B)g') = Ch(h+2xh')}), we obtain
\begin{equation}
\label{eq:g=D(h+2xh'), h = E(g+2(x-a)g')}
g=D(h+2xh'), h = E(g+2(x-B)g')
\end{equation}
 for some $D,E\in K$. Set $n=\frac{d-1}{2}$. Using (\ref{eq:f = A(x-B)g^2}) and (\ref{eq:f-B=Cxh^2}), we can assume without any loss that $\frac{\ell_g}{\ell_h} = \sqrt{\frac{C}{A}}$. Using this in (\ref{eq:Ag(g+2(x-B)g') = Ch(h+2xh')}), direct computation leads to
\begin{equation}
	\label{eq:comparisonOfConstants}
	D=\frac{1}{2n+1}\sqrt{\frac{C}{A}}, E=\frac{1}{2n+1}\sqrt{\frac{A}{C}}.
\end{equation}
 Using these expressions for $D$ and $E$ in (\ref{eq:g=D(h+2xh'), h = E(g+2(x-a)g')}), we obtain
 \begin{equation}
 \label{eq:g in terms of h}
 (2n+1)g = \sqrt{\frac{C}{A}} (h+2xh')
 \end{equation}
 and
 \begin{equation}
 \label{eq:h in terms of g}
 (2n+1)h = \sqrt{\frac{A}{C}} (g+2(x-B)g').
 \end{equation}
 Using (\ref{eq:h in terms of g}) in (\ref{eq:g in terms of h}), and simplifying yields
 \begin{equation}
 \label{eq:g-g'-g''}
((2n+1)^2-1)g = (8x-2B)g' + (4x^2-4Bx)g'',
 \end{equation}
 which is a second order, linear differential equation. Let $g = \sum_{i=0}^{n} a_ix^i$. Comparing the coefficients of both sides in (\ref{eq:g-g'-g''}) and simplifying leads to
 \begin{equation}
 \label{eq:recursive-formula}
i(2i-1)Ba_i = -2(n-i+1)(n+i)a_{i-1}
 \end{equation}
 for $i=1,2,\dots, n$. Finally, recalling that $f(0) = B$, and using this in (\ref{eq:f = A(x-B)g^2}) gives $g(0) = a_0 =\pm \sqrt{-\frac{1}{A}}$, completing the proof of \emph{only if} part of the theorem.\par 
Conversely, by the first part of the proof, conditions (1) and (2) clearly imply that $f$ is not dynamically $2$-ordinary, finishing the proof of Theorem~\ref{thm:strong_odd}.
\end{proof}
\begin{proof}[Proof of Theorem~\ref{thm:2-ordinary}]
	This now immediately follows from Corollary~\ref{cor:finite_case}, Theorem~\ref{thm:strong_even}, and Theorem~\ref{thm:strong_odd}.
	
\end{proof}
Next, we will show that the polynomials arising in Theorem~\ref{thm:strong_odd} (2) are $K$-conjugate to $-T_d$ when $\text{char}(K)=0$ or $\text{char}(K)\geq d$. We will first state a lemma, which will be crucial in the proof. From now on, we will use another normalization of Chebyshev polynomials, given by $\tilde{T}_d(x) = \frac{1}{2}T_d(2x)$. By following the notation in \cite{grubb14}, let $\psi_1(x) = x-2$, $\psi_2(x) = x+2$, and
\[\psi_n(x) = \prod_{\substack{\gcd(k,n)=1 \\ 0<k<\frac{n}{2}}} \bigg(x-2\cos (\frac{2\pi k}{n})\bigg)\]
for $n>2$. It is easy to see that $\psi_n$ has integer coefficients for all $n$ (\cite[Theorem 2.1]{grubb14}). Therefore, the polynomials $\psi_n$ can be taken over any field.
 \begin{lemma}[\cite{grubb14}]
\label{lem:chebyshev}
Let $d\geq 1$ be odd. Then we have
\begin{enumerate}
	\item $\tilde{T}_d(x)-2 = \psi_1 (x)\bigg(\prod_{\substack{k|d\\ k\neq 1}} \psi_k(x)\bigg)^2$.
	\item  $\tilde{T}_d(x)+2 = \psi_2(x)\bigg(\prod_{\substack{k|d\\ k\neq 1}} \psi_{2k}(x)\bigg)^2$.
\end{enumerate}
\end{lemma}
\begin{prop}
	\label{prop:PCF2}
	Let $K$ be a field such that either $\text{char}(K)=0$ or $\text{char}(K)\geq d$ and $\text{char}(K)$ is odd. Suppose that $f\in K[x]$ is defined as follows: $f=A(x-B)g^2$, where $A, B\in \vF_q$, and $g=\sum_{i=0}^{n}a_ix^i$ satisfies
	\[\pm a_0 = \sqrt{-\frac{1}{A}}, \text{ }i(2i-1)Ba_i = -2(n-i+1)(n+i)a_{i-1}\]
	for $i=1,2,\dots, n$. Then $f$ is $K$-conjugate to $-T_d(x)$.
\end{prop}
\begin{proof}
Let $b_i:= \frac{a_i}{a_0}B^i\sqrt{-1}$. Recalling $a_0=\pm\sqrt{-\frac{1}{A}}$, this immediately yields
\[f = B(\frac{x}{B}-1)\bigg(\sum_{i=0}^{n} b_i\big(\frac{x}{B}\big)^i\bigg)^2,\]
which we observe being independent of $A$, so we can select $A=-1$ without changing $f$. Since we only consider the question up to $K$-conjugacy, and $f$ is $K$-conjugate to $h:=(x-1)(\sum_{i=0}^{n} b_ix^i)^2$ for any choice of $B$ (by $\gamma(x)=Bx\in \text{PGL}(2,K)$), we conclude that $B$ can be freely chosen and this does not affect the conjugacy class of $f$.

Now, observe that replacing $x$ with $-(x+2)$ in Lemma~\ref{lem:chebyshev} (1) and Lemma~\ref{lem:chebyshev} (2), we obtain
\begin{equation}
\label{eq:cheb1}
\tilde{T}_n(-(x+2))-2 = -(x+4)(g_n(-(x+2)))^2
\end{equation} 
and
\begin{equation}
\label{eq:cheb2}
\tilde{T}_n(-(x+2))+2 = -x(g_n(-(x+2)))^2.
\end{equation}
Letting $F_n=\tilde{T}_n(-(x+2))-2$, $G_n=g_n(-(x+2))$, $A=-1$, $B=-4$, $C=-1$, Equations (\ref{eq:cheb1}) and (\ref{eq:cheb2}) yield
\[F_n=A(x-B)G_n^2\]
and
\[F_n-B = CxG_n^2.\]
Therefore, by (\ref{eq:f = A(x-B)g^2}) and (\ref{eq:f-B=Cxh^2}), we conclude that $F_n$, the $K$-conjugate of $\tilde{T}_n(-x)=-\tilde{T}_n(x)$ by $\gamma=x+2\in \text{PGL}(2,K)$, is a solution to the differential equation in (\ref{eq:g-g'-g''}) for the specialization $B=-4$. By the form of the solutions of the differential equation and the assumption on $\text{char}(K)$, it is clear that the polynomial solution is unique up to a constant multiple. Hence, we conclude $F_n = Df$ for some $D\in K$. Now, the constant coefficient of $f$ is -4 by the values of $a_0, A$ and $B$. Moreover, the constant coefficient of $F_n$ is also $-4$ by Lemma~\ref{lem:chebyshev} (2). Hence, $D=1$, and $F_n=f$. We conclude that $f$ is $K$-conjugate to $-T_d(x)$, as desired.

\end{proof}

\section{Square patterns in orbits}
\label{sec:OrbitSize}
In this section, we will give an application of our earlier results to a problem related to squares in orbits of polynomials over finite fields. We start by recalling the Weil's bound on multiplicative character sums, which will play a fundamental role in our proofs.
\begin{thm}[Weil Bound on Character Sums]
	Let $\chi$ be the quadratic character on $\vF_q$ defined by 
	$\chi(a)=\left(\frac{a}{q}\right)$ for any $a\in \vF_q$. For $f\in \mathbb{F}_q[x]$, set $d=\text{deg}(f)$. If $f$ is not a perfect square in $\vF_q[x]$, then
	\[
	|\sum_{x\in \mathbb{F}_q} \chi(f(x))| \leq (d-1)\sqrt{q}.
	\]
	
\end{thm}

\begin{notation}
For any $a\in \mathbb{F}_q$ and $f\in \mathbb{F}_q[x]$, we let $s_a$ be the sequence $\{\chi(f^{\ell}(a))\}_{\ell\geq 1}$. We also denote by $s_a(\ell)$ the $\ell$-th element of this sequence.
\end{notation}
We are now ready to prove our first main result.
\begin{thm}
\label{thm:asym_orbit}
	Let $q$ be an odd prime power and $f\in \vF_q[x]$ be a dynamically $2$-ordinary polynomial of degree $d$. Let $a\in \mathbb{F}_q$. Suppose that the sequence $s_a$ is periodic, and let $m:=m_a$ be its period. Then 
	\[|\mathcal{O}_f(a)|=O\left(mq^\frac{2\log_2(d)+1}{2\log_2(d)+2}\right),\]
	where the implied constant is only dependent on $d$.
\end{thm}

\begin{remark}
	Theorem~\ref{thm:asym_orbit} implies that if the orbit is large, then the sequence of squares cannot obey a recurrence of low order. For example, if all elements in $\mathcal{O}_f(a)$  are squares (so $m=1$), and $f$ has degree $2$, then we have $|\mathcal{O}_f(a)|=O(q^\frac{3}{4}) $
	
	Now, suppose for example that $f$ is a permutation polynomial, and the orbit of an element is roughly $Cq$ for $0<C\leq 1$. Then, the period $m$ of the sequence $s_a$ must be at least $Dq^{\varepsilon}$. In other words, we cannot have a structured (in terms of the squareness of the elements of the orbit) large orbit.
\end{remark}


\begin{proof}[Proof of Theorem \ref{thm:asym_orbit}]

Let $L$ be a positive integer.
First, observe that $\mathcal{O}_f(a)$ is endowed with a natural graph structure given by the action of $f$. We will refer to this graph as the functional graph of $f$ starting at $a$. We want to count separately the elements of $\mathcal{O}_f(a)$ that are 
within distance $L$ from the $0$ in functional graph of $f$ starting at $a$ and all the other elements of $\mathcal{O}_f(a)$.
Denote by
\[\mathcal{O}_{f,0,L}(a)=\{x\in \mathcal{O}_f(a): \; f^i(x)=0\; \text{ for some }i\in \{0,\dots, L\}\}.\]
Clearly, $\mathcal{O}_{f,0,L}(a)\leq 2L+1 $. Let $\mathcal{O}_{f,0,L}(a)^c$ be the complement of $\mathcal{O}_{f,0,L}(a)$ in $\mathcal{O}_f(a)$.
Let us denote by $s_a(\ell)$ the $\ell$-th element in the sequence $s_a$.
Our purpose is to provide an estimate for $|\mathcal{O}_f(a)|$ using a suitable character sum:
\[|\mathcal{O}_f(a)|\leq  2L+1 +\sum^{m-1}_{i=0}\sum_{x\in \mathcal{O}_{f,0,L}(a)^c} \prod^L_{\ell=1}{\left(\frac{1+s_a(\ell+i)\chi(f^\ell(x))}{2}\right)}.\]
Let us explain the upper bound above. First, $\mathcal{O}_f(a)$ is split into $\mathcal{O}_{f,0,L}(a)$ and its complement, which leads to the $2L+1$ summand. Now, we need to observe that the sum counts at least one for every element of $ \mathcal{O}_{f,0,L}(a)^c$: If $x \in \mathcal{O}_{f,0,L}(a)^c$, then $f^\ell(x)\neq 0$ for all $\ell \leq L$ and \[\frac{1+s_a(\ell+i)\chi(f^\ell(x))}{2}\] is $1$ for at least one $i\in \{0,\dots, m-1\}$ since the sign of  $s_a(\ell+i)$ will agree with the sign of $\chi(f^\ell(x))$ for at least one $i\leq m-1$, as the period of the sequence $s_a$ is $m$.

Now, since $\prod^L_{\ell=1}{\left(\frac{1+s_a(\ell+i)\chi(f^\ell(x))}{2}\right)}$ is positive for all $x\in \vF_q$, we can upper bound again by summing over all $x$'s:
\[|\mathcal{O}_f(a)|\leq  2L+1 +\sum^{m-1}_{i=0}\sum_{x\in \vF_q} \prod^L_{\ell=1}{\left(\frac{1+s_a(\ell+i)\chi(f^\ell(x))}{2}\right)}.\]

%
%
%
%
%
%
%
%
%

	Using ideas similar to \cite{heath2019irreducible,ostafe2010length}, we are going to bound 
	\[B_i:=\sum_{x\in \vF_q} \prod^L_{\ell=1}{\left(\frac{1+s_a(\ell+i)\chi(f^\ell(x))}{2}\right)}\]
	uniformly in $i$ with a certain $B$, from which it will follow that
	\begin{equation}
	\label{eq:OrbitBound}
	|\mathcal{O}_f(a)|\leq 2L+1+mB.
	\end{equation}

	Observe that 
	\[\prod^L_{\ell=1}{\left(\frac{1+s_a(\ell+i)\chi(f^\ell(x))}{2}\right)}=\frac{1}{2^L}\sum_{T\subseteq \{1,\dots, L\} }\prod_{t\in T} s_a(t+i)\chi(f^t(x)). \]
	Exchanging the sums in $x$ and $T$ yields
	\[B_i= \frac{q}{2^L}+\frac{1}{2^L}\sum_{\emptyset \subsetneq T\subseteq \{1,\dots, L\} }\sum_{x\in \vF_q }\prod_{t\in T} s_a(t+i)\chi(f^t(x)),\]
	which can be rewritten as
	\[B_i= \frac{q}{2^L}+\frac{1}{2^L} \sum_{\emptyset \subsetneq T\subseteq \{1,\dots, L\} }\left(\prod_{t\in T} s_a(t+i) \right)\sum_{x\in \vF_q }\prod_{t\in T}\chi(f^t(x)),\]
	where $\prod_{t\in T} s_a(t+i) \in\{-1,1,0\}$. Therefore,
	\[B_i\leq  \frac{q}{2^L}+\frac{1}{2^L} \sum_{\emptyset \subsetneq T\subseteq \{1,\dots, L\} }\left|\sum_{x\in \vF_q }\chi\left(\prod_{t\in T}f^t(x)\right)\right|.\]
	Note that $\prod_{t\in T}f^t(x)$ is never a square in $\mathbb{F}_q[x]$ because by assumption, for all $t\in T$, there is always a new irreducible factor dividing $f^t$ with odd multiplicity. Therefore, we can use Weil's bound to get
	\[B_i\leq  \frac{q}{2^L}+\frac{1}{2^L} \left(\frac{d^{L+1}-1}{d-1}\right)\sqrt{q}\sum_{\emptyset \subsetneq T\subseteq \{1,\dots, L\} }1=\frac{q}{2^L}+ \left(\frac{d^{L+1}-1}{d-1}\right)\sqrt{q}.\]
	
	It follows that
	\begin{equation}\label{eq:bounddepL}
	B_i\leq \frac{q}{2^L}+d^{L+1}\sqrt{q}.
	\end{equation}
	Note that the above inequality is true for all $L$, so we need to optimize $L$ so that the inequality becomes the strongest possible. For that, we need to fix $L$ such that 
	$\frac{q}{2^L}\sim d^{L+1}\sqrt{q}$
	as $q$ grows. Therefore, we need
	\[\lim_{q\rightarrow \infty} \frac{q}{2^L d^{L+1}\sqrt{q}}=\lim_{q\rightarrow \infty} \frac{\sqrt{q}}{2^{L+(L+1)\log_2(d)}}=\lim_{q\rightarrow \infty} \frac{\sqrt{q}}{2^{L(\log_2(d)+1)+\log_2(d)}}=1.\]
	A direct computation yields the choice $$L\sim\log_2(\sqrt{q}/d)/(\log_2(d)+1).$$ Therefore, by directly computing $q/2^L$ with this choice of $L$, we obtain
	\[B=O\left(q^\frac{2\log_2(d)+1}{2\log_2(d)+2}\right).\]
	Combining this and the choice of $L$ with \eqref{eq:OrbitBound} completes the proof.
	
\end{proof}

\begin{remark}
Notice that the constant in the $O$ can be made explicit by simply choosing $L$ to be the floor of $\log_2(\sqrt{q}/d)/(\log_2(d)+1)$ and replacing it in the inequality \eqref{eq:bounddepL}.
\end{remark}

\begin{proof}[Proof of Theorem \ref{thm:orbitsize_explicit}]
This is an immediate consequence of Corollary~\ref{cor:finite_case}, Theorem~\ref{thm:strong_even}, Theorem~\ref{thm:strong_odd}, Proposition~\ref{prop:PCF1}, Proposition~\ref{prop:PCF2}, and Theorem~\ref{thm:asym_orbit}.
\end{proof}


Using a similar method, we next prove an asymptotic result on the maximum length of consecutive squares in dynamical orbits over finite fields. 
\begin{thm}
\label{thm:consecutive_squares_implicit}
	Let $q$ be an odd prime power and $f\in \vF_q[x]$ be a dynamically $2$-ordinary polynomial of degree $d$. Let $a\in \mathbb{F}_q$.  Then the longest sequence of consecutive squares (or non-squares) in $\mathcal{O}_f(a)$ has length $O\left(q^\frac{2\log(d)+1}{2\log(d)+2}\right)$, where the implied constant is only dependent on $d$.
\end{thm}
\begin{proof}
We can restrict to the case of squares, as the case of non-squares is completely analogous. 
Suppose that there is a set $Y$ of $4S+1$ consecutive squares in $\mathcal{O}_f(a)$, which means that there is an element $b\in \mathcal{O}_f(a)$ such that $f^i(b)$ is a square for every $i\in \{1,\dots, 4S+1\}$.
Consider two subsets $Y_1,Y_2$ of $Y$, which we now describe:  Let $Y_1$ be the set of elements of $Y$ which have distance at most $S$ from the zero in the functional graph of $f$ starting at $a$. 
Note that $Y_1=\emptyset$ if the zero does not lie in $\mathcal{O}_f(a)$, and in general, we have $|Y|\leq 2S+1$. Moreover, set 
\[Y_2=\{f^i(b): i\in \{3S+2,\dots, 4S+1\}\}\]
The set $Y_2$ consists of the only elements of $Y$ such that if one applies $f$ a number of times less than or equal to $S$, one could land outside $Y$ (these are essentially the ``last'' $S$ elements of $Y$).
Set $X=Y\setminus (Y_1\cup Y_2)$ and observe that $X$ has size at least $S$. Now, all elements $x\in X$ satisfy the following properties:
\begin{itemize}
\item $f^i(x)$ is  non-zero for all $i\in \{1,\dots, S\}$ because $x\not\in Y_1$
\item $f^i(x)$ is  a square for all $i\in \{1,\dots, S\}$ because it is not in $Y_2$ and therefore $f^i(x)$ must land in $Y$.
\end{itemize}

So we established that for some $c\in \mathcal O_f(a)$ there is a set $X'\subseteq X$ of $S$  elements \[\{f(c), f^2(c), \dots f^{S}(c)\}\] that are all squares, and such that $f^{i}(x)\neq 0 $ and $f^i(x)$ is a square for any $x\in X'$ for all $i\in \{1,\dots,S\}$. Now, choose $L<S$ and consider the set $T(L)$ of elements $x\in \vF_q$ such that $f^i(x)$ is a square for all $i\in \{1,\dots,L\}$, and $f^i(x)\neq 0$. Clearly $X'\subseteq T(L)$, so if we can provide an estimate for $T(L)$, we can just tune $L$ to minimize $|T(L)|$ and get the strongest possible bound on $S$ (and therefore on $4S+1$).
Observe now that 
\[|T(L)|=\sum_{x\in T(L)} \prod^L_{\ell=1}{\left(\frac{1+\chi(f^\ell(x))}{2}\right)}\]
because the internal product in the sum is counting one for every element in $T(L)$. By extending the summation index to $\vF_q$ and observing that we are only adding positive or zero values, we get 
\[|T(L)|\leq \sum_{x\in \vF_q} \prod^L_{\ell=1}{\left(\frac{1+\chi(f^\ell(x))}{2}\right)}.\]
Now argue exactly as in the estimate for $B_i$ in the proof of Theorem \ref{thm:asym_orbit} with the $s_a$ consisting of the constant sequence of ones and choose again $L\sim\log(\sqrt{q}/d)/(\log(d)+1)$, getting $S=O(q^\frac{2\log(d)+1}{2\log(d)+2})$ and concluding the proof.
It is worth noting that if the choice $L\sim\log(\sqrt{q}/d)/(\log(d)+1)$ is not possible, the claim is immediately true because then we would have $S=O(\log(\sqrt{q}/d)/(\log(d)+1))$, which would give a tighter bound.
For the case of non-squares, simply replace $\frac{1+\chi(f^\ell(x))}{2}$ with $\frac{1-\chi(f^\ell(x))}{2}$.
\end{proof}

\begin{proof}[Proof of Theorem \ref{thm:consecutive_squares_explicit}]
This immediately follows by combining Corollary~\ref{cor:finite_case}, Theorem~\ref{thm:strong_even}, Theorem~\ref{thm:strong_odd}, Proposition~\ref{prop:PCF1}, Proposition~\ref{prop:PCF2}, and Theorem~\ref{thm:consecutive_squares_implicit}.
\end{proof}

\end{document}